\documentclass{amsart}
\usepackage{amssymb}
\usepackage[top= 1.5in, bottom=1.5in, left=1.5in, right=1.5in]{geometry}
\usepackage{mathrsfs}
\usepackage{graphicx}
\usepackage{comment}
\allowdisplaybreaks

\newcommand{\be}{\begin{equation}}
\newcommand{\ee}{\end{equation}}
\newcommand{\ba}{\begin{align}}
\newcommand{\ea}{\end{align}}

\newtheorem{theorem}{Theorem}[section]
\newtheorem{lemma}[theorem]{Lemma}
\newtheorem{proposition}[theorem]{Proposition}
\newtheorem*{proposition*}{Proposition}
\newtheorem{corollary}[theorem]{Corollary}
\newtheorem*{theorem*}{Theorem}
\newtheorem*{corollary*}{Corollary}
\newtheorem*{cor*}{Corollary}

\theoremstyle{definition}

{\begin{list}{}{%
\settowidth{\labelwidth}{\textsf{{\it #1.}}}%
\setlength{\labelsep}{4mm}%
\setlength{\leftmargin}{\labelwidth}%
\addtolength{\leftmargin}{\labelsep}%
}}%
{\end{list}}

\def\beq{\begin{equation}}\def\enq{\end{equation}}

\def\myalpha{R(q)}
\def\mybeta{R(q^2)}
\def\mygamma{R(q^3)}

{\begin{list}{}{%
\settowidth{\labelwidth}{\textsf{{\it #1.}}}%
\setlength{\labelsep}{2mm}%
\setlength{\leftmargin}{\labelwidth}%
\addtolength{\leftmargin}{\labelsep}%
\addtolength{\leftmargin}{4mm}%
\setlength{\itemsep}{6pt}%
\setlength{\listparindent}{0pt}%
\setlength{\topsep}{3pt}%
}}%

\usepackage{color}
\usepackage[normalem]{ulem}
\usepackage{soul}

\usepackage{xcolor}

\newcommand{\tbar}{\overline{t}}

\begin{document}

\title[Congruences for overpartitions with restricted odd differences]{Elementary proofs of congruences modulo 5 for overpartitions with restricted odd differences}

\author[B. Paudel]{Bishnu Paudel}

\author[J. Sellers]{James A. Sellers}

\author[H. Wang]{Haiyang Wang}

\address{Mathematics and Statistics Department\\
         University of Minnesota Duluth\\
         Duluth, MN 55812, USA}
\email{bpaudel@d.umn.edu, jsellers@d.umn.edu, wan02600@d.umn.edu}

\begin{abstract}
In 2015, Bringmann, Dousse, Lovejoy, and Mahlburg defined the function $\overline{t}(n)$ to be the number of overpartitions of weight $n$ where (i) the difference between two successive parts may be odd only if the larger part is overlined and (ii) if the smallest part is odd then it is overlined.  In their work, they proved that $\overline{t}(n)$ satisfies an elegant congruence modulo 3.  Since then, a number of authors have studied arithmetic properties satisfied by $\overline{t}(n)$.  In particular, in 2023, Hanson and Smith utilized the theory of modular forms to prove the following two congruences modulo 5:  For all $n\geq 0$, 
\begin{align*}
\overline{t}(80n+40)\equiv \overline{t}(80n+60)\equiv 0 \pmod{5}.
\end{align*}
Our goal in this work is to provide a truly elementary proof of this pair of congruences.  
\end{abstract}

\maketitle

\section{Introduction}\label{sec.intro}
A partition $\lambda$ of a positive integer $n$ is a nonincreasing sequence of positive numbers $\lambda_1, \lambda_2, \dots, \lambda_r$ such that $\lambda_1 + \lambda_2 + \dots + \lambda_r = n$.  An overpartition of $n$ is a partition of $n$ wherein the last instance of any given part size may (or may not) be overlined.  For example, the partitions of $n=4$ are given by 
$$
4, \ \ \ 3+1, \ \ \ 2+2, \ \ \ 2+1+1, \ \ \ 1+1+1+1,
$$
while the overpartitions of $n=4$ are 
$$
4, \ \ \ \overline{4}, \ \ \ 3+1, \ \ \ \overline{3}+1, \ \ \ 3 + \overline{1}, \ \ \ \overline{3}+\overline{1}, \ \ \ 2+2, \ \ \ 2+\overline{2}, 
$$
$$2+1+1, \ \ \ \overline{2}+1+1, \ \ \ 2 + 1+\overline{1}, \ \ \ \overline{2}+1+\overline{1}, \ \ \ 1+1+1+1, \ \ \ 1+1+1+\overline{1}.
$$
In 2015, Bringmann, Dousse, Lovejoy, and Mahlburg \cite{BDLM} defined the function $\overline{t}(n)$ to be the number of overpartitions of weight $n$ where (i) the difference between two successive parts may be odd only if the larger part is overlined and (ii) if the smallest part is odd then it is overlined.  For example, $\overline{t}(4)= 8$ where the overpartitions in question are   
$$4,  \ \ \  \overline{4}, \ \ \ 3+\overline{1}, \ \ \  \overline{3}+\overline{1},  \ \ \ 2+2,  \ \ \ 2+\overline{2},  \ \ \ \overline{2}+1+\overline{1},  \ \ \  1+1+1+\overline{1}.
$$ 
In \cite{BDLM}, the authors prove that the generating function for $\overline{t}(n)$ is given by 
\begin{equation}\label{generating}
\sum_{n\ge0}\tbar(n)q^{n}=\frac{f_3}{f_1f_2},
\end{equation}
where 
$$
f_k=f_k(q) :=(q^k;q^k)_\infty\qquad\text{with} \qquad(a;q)_\infty:=\prod\limits_{i\geq 0}(1-aq^{i}).
$$
They also proved that $\overline{t}(n)$ satisfies a nice characterization modulo 3.  
\begin{theorem}
%%%%%%%%%%%%%%%%%%%%%
\label{3characterization}
%%%%%%%%%%%%%%%%%%%%%
For all $n\geq 1,$ 
\begin{equation*}
\overline{t}(n) \equiv
\begin{cases} 
(-1)^{k+1}\phantom{0}  \pmod{3} & \mbox{if } n =k^2 \mbox{ for\ some\ integer\ } k, \\
0\phantom{(-1)^{k+1}} \pmod{3} & \mbox{otherwise. } 
\end{cases} 
\end{equation*}
\end{theorem}
Since the publication of \cite{BDLM}, several other authors have proven arithmetic properties satisfied by $\overline{t}(n)$; see \cite{CH2019, GJ2025, HS2023, HS2020, LLWX, Zhang2025, Zhang2026} for examples.  In particular, in 2023, Hanson and Smith \cite{HS2023} utilized the theory of modular forms to prove the following two congruences modulo 5:  For all $n\geq 0$, 
\begin{equation}
\label{HansonSmith_mod5congruences}
\overline{t}(80n+40)\equiv \overline{t}(80n+60)\equiv 0 \pmod{5}.
\end{equation}
Our goal in this work is straightforward -- to provide a truly elementary proof of the pair of congruences in \eqref{HansonSmith_mod5congruences} which does not rely in any way on modular forms.  Instead, we utilize well--known dissection properties which already appear in the literature, along with a generalization of Chern and Tang's work \cite{CT2021}, to initially determine the behavior of the generating function for $\overline{t}(5n)$ modulo 5.  This is, by far, the most involved aspect of our proof.  Once this is complete, we simply need to 2--dissect the generating function for $\overline{t}(5n)$ further until we reach a congruence modulo 5 satisfied by the generating function for $\overline{t}(20n)$.  The two 2--dissections that are necessary here are extremely elementary and are completed with ease.  Once this is done, the proof of \eqref{HansonSmith_mod5congruences} is readily finished. 

Said in more detail, our primary objective in this paper is the following theorem.  
\begin{theorem}\label{thm.20n}
We have
\begin{equation}
\sum_{n\ge0}\tbar(20n)q^{n}\equiv Y_0(q^{4})+qY_1(q^{4})\pmod5,
\label{eq.main80}
\end{equation}
where
\begin{equation*}
Y_0(q):=\frac{f_3^{5}f_4}{f_1f_6f_{12}}
+q\frac{f_2^{7}f_{12}^{3}}{f_1^{2}f_4^{3}f_6^{2}} \qquad \text{and}\qquad 
Y_1(q):=4\frac{f_1f_4^{3}f_6^{7}}{f_2^{2}f_3^{3}f_{12}^{3}}
+4\frac{f_2^{3}f_3^{6}f_{12}}{f_1^{2}f_4f_6^{4}}.
\end{equation*}
\end{theorem}
With Theorem \ref{thm.20n} in hand, the following corollary immediately completes our proof of \eqref{HansonSmith_mod5congruences}.  
\begin{corollary}\label{cor.80}
For all $n\ge0$,
\[
\tbar(80n+40)\equiv0\pmod5
\qquad\text{and}\qquad
\tbar(80n+60)\equiv0\pmod5 .
\]
\end{corollary}
\begin{proof}
Note that, when written as a power series in $q$, the right--hand side of the congruence \eqref{eq.main80} contains no powers of the form $q^{4n+2}$ or $q^{4n+3}$.  Therefore,
\[
\tbar(80n+40)=\tbar\bigl(20(4n+2)\bigr)\equiv0
\quad\text{and}\quad
\tbar(80n+60)=\tbar\bigl(20(4n+3)\bigr)\equiv0\pmod5. \qedhere
\]\end{proof}

The rest of this paper is organized as follows. In Section \ref{sec.prelim},  we collect a number of necessary mathematical tools. We provide an elementary proof of Theorem \ref{thm.20n} in Section \ref{sec.proof}. We conclude the paper in Section \ref{sec.ct} by providing a new elementary proof of a congruence established by Chern and Hao \cite{CH2019} and by Lin et al. \cite{LLWX}.

\section{Preliminaries}\label{sec.prelim}
We require the following 2-dissection results in the work below.
\begin{lemma} We have
\begin{align}
f_1^{2}
&=\frac{f_2f_8^{5}}{f_4^{2}f_{16}^{2}}-2q\,\frac{f_2f_{16}^{2}}{f_8},
\label{eq.f1sq}\\
\frac{f_1^{3}}{f_3}
&=\frac{f_4^{3}}{f_{12}}
-3q\frac{f_2^{2}f_{12}^{3}}{f_4f_6^{2}}
\label{eq.f13f3},\\
\frac{f_3^{3}}{f_1}
&=\frac{f_4^{3}f_6^{2}}{f_2^{2}f_{12}}+q\,\frac{f_{12}^{3}}{f_4},
\label{eq.f33f1}\\
\frac{f_1}{f_3}
&=\frac{f_2f_{16}f_{24}^{2}}{f_6^{2}f_8f_{48}}
-q\,\frac{f_2f_8^{2}f_{12}f_{48}}{f_4f_6^{2}f_{16}f_{24}},
\label{eq.f1f3}\\
\frac1{f_1f_3}
&=\frac{f_8^{2}f_{12}^{5}}{f_2^{2}f_4f_6^{4}f_{24}^{2}}
+q\,\frac{f_4^{5}f_{24}^{2}}{f_2^{4}f_6^{2}f_8^{2}f_{12}}
\label{eq.1f1f3},\\
\frac{f_1^{2}}{f_3^{2}}
&=\frac{f_2f_4^{2}f_{12}^{4}}{f_6^{5}f_8f_{24}}
-2q\frac{f_2^{2}f_8f_{12}f_{24}}{f_4f_6^{4}}\label{eq.f12f32}.
\end{align}
\end{lemma}
\begin{proof}
Equation \eqref{eq.f1sq} is (18) in \cite[Lemma 1]{daSilva-Sellers}. Identities \eqref{eq.f13f3} and \eqref{eq.f12f32} correspond to (30) and (34), respectively, in \cite[Lemma 1]{SS20}. Equations \eqref{eq.f33f1} and \eqref{eq.f1f3} appear as (3.75) and (2.17) in \cite{XiaYao} and \cite{XiaYaob}, respectively. Identity \eqref{eq.1f1f3} follows by adding Equations (4.2) and (4.3) given in \cite[Theorem 4.3]{BaruahOjah}.
\end{proof}
\begin{lemma} We have 
\begin{equation}
\frac{f_1^{6}}{f_3^{2}}
=\frac{f_2f_4f_6^{5}}{f_{12}^{3}}
-6q\frac{f_2^2f_4^2f_{12}^2}{f_6^2}
+10q^{2}\frac{f_2^{4}f_{12}^{6}}{f_4^{2}f_6^{4}}.
\label{eq.f16f32}
\end{equation}
\end{lemma}
\begin{proof} Squaring \eqref{eq.f13f3}, we get 
\begin{equation}
\label{eq.f16f321}\frac{f_1^{6}}{f_3^{2}}
=\left(\frac{f_4^3}{f_{12}}\right)^2-\left(q\frac{f_2^{2}f_{12}^{3}}{f_4f_6^{2}}\right)^2
-6q\frac{f_2^2f_4^2f_{12}^2}{f_6^2}
+10q^{2}\,\frac{f_2^{4}f_{12}^{6}}{f_4^{2}f_6^{4}}.\end{equation}
Multiplying \eqref{eq.f33f1} by $f_2^2/f_6^2$ gives 
\begin{equation}\label{+}
    \frac{f_4^3}{f_{12}}+q\frac{f_2^{2}f_{12}^{3}}{f_4f_6^{2}}=\frac{f_2^2f_3^3}{f_1f_6^2}.
\end{equation}
Now, replacing $q$ by $-q$ in \eqref{+} and using the fact $f_1(-q)=\dfrac{f_2^3}{f_1 f_4}$ yields
\begin{align}\label{-}
    \frac{f_4^3}{f_{12}}-q\frac{f_2^{2}f_{12}^{3}}{f_4f_6^{2}}&=\frac{f_2^2}{f_6^2}\cdot f_3^3(-q)\cdot \frac{1}{f_1(-q)}=\frac{f_1f_4f_6^{7}}{f_2f_3^{3}f_{12}^{3}}.
\end{align}
Using \eqref{+} and \eqref{-}, we obtain $$\left(\frac{f_4^3}{f_{12}}\right)^2-\left(q\frac{f_2^{2}f_{12}^{3}}{f_4f_6^{2}}\right)^2=\left(\frac{f_4^3}{f_{12}}+q\frac{f_2^{2}f_{12}^{3}}{f_4f_6^{2}}\right)\left(\frac{f_4^3}{f_{12}}-q\frac{f_2^{2}f_{12}^{3}}{f_4f_6^{2}}\right)=\frac{f_2f_4f_6^{5}}{f_{12}^{3}}.$$
Substituting the last identity into \eqref{eq.f16f321} completes the proof.
\end{proof}
Next, we introduce the Rogers--Ramanujan continued fraction
\[
R:=R(q)=\frac{(q;q^5)_\infty(q^4;q^5)_\infty}{(q^2;q^5)_\infty(q^3;q^5)_\infty},
\]
which first appeared in the work of Rogers \cite{Rogers} and later in Ramanujan's first letter to Hardy dated January 16, 1913. For further background and a detailed treatment, see \cite[Chs.~8, 9, 15, 16, and~17]{Hirschhorn}. With the continued fraction $R$, we have the following $5$-dissections.

\begin{lemma}\label{lemma5-disse} We have
\begin{align}
    f_1&=f_{25}\left(\frac{1}{R(q^5)}-q-q^2R(q^5)\right), \label{eq.f1}\\
    \frac{1}{f_1}&=\frac{f_{25}^{\,5}}{f_{5}^{\,6}}\,\Phi\!\bigl(R(q^{5}),\,q\bigr), \label{eq.1f1}
\end{align}
where
\begin{equation*}
\Phi(r,Q)=\frac1{r^{4}}+\frac{Q}{r^{3}}+\frac{2Q^{2}}{r^{2}}+\frac{3Q^{3}}{r}+5Q^{4}
-3Q^{5}r+2Q^{6}r^{2}-Q^{7}r^{3}+Q^{8}r^{4}.
%\label{eq.Phi}
\end{equation*}
\end{lemma}
\begin{proof}
    Identity \eqref{eq.f1} was stated by Ramanujan without proof in his work on continued fractions, and a proof can be found in \cite{Watson}. Equation \eqref{eq.1f1} appears as (7.4.14) in \cite{Berndt_2}.
\end{proof}

As immediate consequences of the previous lemma, we obtain
\begin{align}
\label{eq.f3}
f_3&=f_{75}\left(\frac1{R(q^{15})}-q^{3}-q^{6}\,R(q^{15})\right),\\
\label{eq.recip} \frac{1}{f_k}&=\frac{f_{25k}^{\,5}}{f_{5k}^{\,6}}\,\Phi\!\bigl(R(q^{5k}),\,q^k\bigr).
\end{align}
To simplify the notation in our work below, we set
\begin{align}
  %&\alpha=R(q), \qquad \qquad\beta=R(q^{2}), \qquad \qquad\gamma=R(q^{3}),\\
 \label{K-X}  &K=q^{-1}\frac{f_2f_5^{5}}{f_1f_{10}^{5}},\qquad
S=q^{-2}\frac{f_1^{3}f_3^{3}}{f_5^{3}f_{15}^{3}},\qquad
T=q^{-2}\frac{f_3f_5^{5}}{f_1f_{15}^{5}},\\
\label{V-def}&U=q^{-1}\frac{f_2^{3}f_3^{3}}
{f_1^{2}f_6^{2}f_{10}f_{15}},\qquad V=\frac{f_1^{2}f_6f_{30}}{f_3f_{10}^{2}f_{15}}.
\end{align}

For $\alpha\in\mathbb{Z}_{\ge0}$ and $\beta\in\mathbb{Z}$, Chern and Tang
\cite{CT2021} introduced the two families
\begin{align*}
	&\frac{1}{q^{\alpha}R(q)^{\alpha+2\beta}R(q^{2})^{2\alpha-\beta}}
	+(-1)^{\alpha+\beta}\,q^{\alpha}R(q)^{\alpha+2\beta}R(q^{2})^{2\alpha-\beta},\\
	&\frac{1}{q^{\alpha}R(q)^{2\alpha+3\beta}R(q^{3})^{\alpha-\beta}}
	+(-1)^{\alpha}\,q^{\alpha}R(q)^{2\alpha+3\beta}R(q^{3})^{\alpha-\beta},
\end{align*}
and represented them in terms of eta quotients involving $K, S, T$ through recurrence relations. In \cite{PaudelSellersWang}, the authors extended the results of Chern and Tang to the following three-parameter family, which is extremely beneficial for representing the generating functions in question. For a triple $(m,n,p)\in\mathbb{Z}^{3}$, we set
\begin{equation*}
w(m,n,p):=q^{-p}\,R(q)^{-2m-n}\,R(q^{2})^{m-n-p}\,R(q^{3})^{n-p},
\end{equation*}
and define
\begin{equation}\label{eq.G}
G(m,n,p):=w(m,n,p)+\frac{(-1)^{m+n}}{w(m,n,p)}.
\end{equation}
Then, we have $G(0,0,0)=2$ and the reflection property
\begin{equation}
G(-m,-n,-p)=(-1)^{m+n}\,G(m,n,p).
\label{eq.reflect}
\end{equation}

\begin{lemma}[{\cite[Theorem~1.1]{PaudelSellersWang}}]\label{thm.P1main}
For all $(m,n,p)\in\mathbb{Z}^{3}$, the following recurrences hold:
\begin{align}
G(m+1,n,p)&=\frac4K\,G(m,n,p)+G(m-1,n,p),
\label{eq.recm}\\
G(m,n+1,p)&=(1-V)\,G(m,n,p)+G(m,n-1,p),
\label{eq.recn}\\
G(m,n,p+1)&=(U-2)\,G(m,n,p)-G(m,n,p-1).
\label{eq.recp}
\end{align}
Moreover, we have the initial values
\begin{align}
\label{eq.init0}G(0,0,0)&=2,\\
G(1,0,0)&=\frac4K,\\
G(0,1,1)&=K,\\
G(1,1,1)&=K+2+\frac4K,\\
G(1,1,0)&=2+\frac9T,\\
G(1,0,1)&=\frac14\left(T-S+\frac9T+6\right),\\
G(0,1,0)&=1-V,\\
G(0,0,1)&=U-2.
\label{eq.init1}
\end{align}
\end{lemma}

\begin{lemma} We have
\begin{equation}
  K-3-\dfrac4K =q^{-1}\dfrac{f_1^{2}f_2^{2}}{f_5^{2}f_{10}^{2}}. \label{eq.k}
\end{equation}
\end{lemma}
\begin{proof}
Identity~\eqref{eq.k} follows from \cite[Eqs.~(1.7) and~(2.2)]{Baruah-Begum}.
\end{proof}

We also require the parameterization of eta quotients due to Alaca, Alaca, and Williams \cite{AAW2006}. Let
\[
\varphi(q):=1+2\sum_{n\ge1}q^{n^{2}}
\]
denote the classical theta function of Ramanujan, and set
\begin{equation}\label{eq.stw}
s:=\frac{\varphi^{3}(q^{3})}{\varphi(q)},\qquad
t:=\frac{\varphi^{2}(q)-\varphi^{2}(q^{3})}{4q\,\varphi^{2}(q^{3})} .
\end{equation}

\begin{lemma}[{\cite{AAW2006}}]\label{lem.f1236}
We have
\begin{align}
\label{eq.aaw1}f_1&=s^{1/2}\,t^{1/24}\,(1-2qt)^{1/2}\,(1+qt)^{1/8}\,(1+2qt)^{1/6}\,(1+4qt)^{1/8},\\
\label{eq.aaw2}f_2&=s^{1/2}\,t^{1/12}\,(1-2qt)^{1/4}\,(1+qt)^{1/4}\,(1+2qt)^{1/12}\,(1+4qt)^{1/4},\\
\label{eq.aaw3}f_3&=s^{1/2}\,t^{1/8}\,(1-2qt)^{1/6}\,(1+qt)^{1/24}\,(1+2qt)^{1/2}\,(1+4qt)^{1/24},\\
\label{eq.aaw6}f_6&=s^{1/2}\,t^{1/4}\,(1-2qt)^{1/12}\,(1+qt)^{1/12}\,(1+2qt)^{1/4}\,(1+4qt)^{1/12}.
\end{align}
\end{lemma}

{Lastly, as we close this section, we note the following lemma which follows from the Binomial Theorem.}
\begin{lemma}
    For a prime $p$ and positive integers $k$ and $l$,
    \begin{equation}\label{bt}
        f_l^{p^k}\equiv f_{lp}^{p^{k-1}}\pmod{p^k}.
    \end{equation}
\end{lemma}

\section{Proof of Theorem \ref{thm.20n}}\label{sec.proof}

We begin by establishing the following generating function result.
\begin{theorem}\label{thm.main}
We have
\begin{equation}
\sum_{n\geq0}\tbar(5n)q^{n}
\equiv\frac{f_1f_2f_3^{3}}{f_6^{2}}\pmod5 .
\label{eq.main}
\end{equation}
\end{theorem}

\begin{proof}
Using \eqref{eq.f3} and \eqref{eq.recip} with $k=1,2$ in \eqref{generating}, we get
\begin{align*}
    \sum_{n\ge0}\tbar(n)q^{n}= \frac{f_{25}^{\,5}f_{50}^{\,5}f_{75}}{f_5^{\,6}f_{10}^{\,6}}\,
\Phi\bigl(R(q^{5}),q\bigr)\,\Phi\bigl(R(q^{10}),q^{2}\bigr)
\left(\frac1{R(q^{15})}-q^{3}-q^{6}\,R(q^{15})\right).
\end{align*}
Extracting the terms in which the exponents of $q$ are congruent to $0$ modulo $5$ and then replacing $q^5$ by $q$ gives
\begin{equation*}
\sum_{n\geq0}\tbar(5n)q^n
=\frac{f_5^{\,5}f_{10}^{\,5}f_{15}}{f_1^{\,6}f_2^{\,6}}\,\Lambda ,
%\label{eq.exact-X}
\end{equation*}
where $\Lambda$ denotes the sum of the terms of the product
$$\Phi\bigl(R(q^{5}),q\bigr)\,\Phi\bigl(R(q^{10}),q^{2}\bigr)
\bigl(R(q^{15})^{-1}-q^{3}-q^{6}R(q^{15})\bigr)$$
whose $q$-exponent is divisible by $5$, with $q^{5}$ then replaced by $q$. Thus $\Lambda$
{can be expressed as} a Laurent polynomial in
%\begin{equation}
%\label{a-r}
%    \myalpha:=R(q), \qquad \qquad\mybeta:=R(q^{2}), \qquad \qquad\mygamma:=R(q^{3})
%\end{equation}
$R(q)$, $R(q^2)$, and $R(q^3)$
with coefficients in
$\mathbb Z[q]$. Carrying out this finite extraction and reducing every coefficient modulo $5$ simplifies
$\Lambda$ to
\begin{equation*}
\Lambda\equiv \frac{L_0}{\mygamma}+L_1+\mygamma\,L_2\pmod5,
%\label{eq.L-form}
\end{equation*}
where $L_0$, $L_1$, $L_2$ are the following Laurent polynomials in $\myalpha$ and $\mybeta$:
\begin{align*}
L_0={}&\myalpha^{-4}\mybeta^{-4}
+q\!\left(2\myalpha\mybeta^{-4}+3\myalpha^{-1}\mybeta^{-3}+2\myalpha^{-3}\mybeta^{-2}\right) \\
&+q^2\!\left(\myalpha^{4}\mybeta^{-3}+4\myalpha^{2}\mybeta^{-2}+2\myalpha^{-4}\mybeta\right)\\
&+q^3\!\left(4\myalpha\mybeta+\myalpha^{-1}\mybeta^{2}+4\myalpha^{-3}\mybeta^{3}\right) \\
&+q^4\!\left(2\myalpha^{4}\mybeta^{2}+3\myalpha^{2}\mybeta^{3}\right),\\[3pt]
L_1={}&q\!\left(3\myalpha^{-2}\mybeta^{-4}+4\myalpha^{-4}\mybeta^{-3}\right)\\
&+q^2\!\left(\myalpha^{3}\mybeta^{-4}+3\myalpha\mybeta^{-3}+4\myalpha^{-1}\mybeta^{-2}+2\myalpha^{-3}\mybeta^{-1}\right)\\
&+q^3\!\left(3\myalpha^{4}\mybeta^{-2}+4\myalpha^{2}\mybeta^{-1}+\myalpha^{-2}\mybeta+3\myalpha^{-4}\mybeta^{2}\right)\\
&+q^4\!\left(2\myalpha^{3}\mybeta+\myalpha\mybeta^{2}+3\myalpha^{-1}\mybeta^{3}+4\myalpha^{-3}\mybeta^{4}\right)\\
&+q^5\!\left(\myalpha^{4}\mybeta^{3}+3\myalpha^{2}\mybeta^{4}\right),\\[3pt]
L_2={}&q^2\!\left(3\myalpha^{-2}\mybeta^{-3}+3\myalpha^{-4}\mybeta^{-2}\right)\\
&+q^3\!\left(\myalpha^{3}\mybeta^{-3}+\myalpha\mybeta^{-2}+\myalpha^{-1}\mybeta^{-1}\right)\\
&+q^4\!\left(2\myalpha^{4}\mybeta^{-1}+\myalpha^{-2}\mybeta^{2}+\myalpha^{-4}\mybeta^{3}\right)\\
&+q^5\!\left(2\myalpha^{3}\mybeta^{2}+2\myalpha\mybeta^{3}+2\myalpha^{-1}\mybeta^{4}\right)\\
&+q^6\!\left(4\myalpha^{4}\mybeta^{4}\right).
\end{align*}
Thus, we obtain
\begin{equation}
\sum_{n\geq0}\tbar(5n)q^n
\equiv\frac{f_5^{5}f_{10}^{5}f_{15}}{f_1^{6}f_2^{6}}\,
\left(\frac{L_0}{\mygamma}+L_1+\mygamma L_2\right)\pmod5.
\label{eq.L-GF}
\end{equation}

We now pair the twelve terms of $L_0/\mygamma$ with the twelve terms of
$\mygamma L_2$, while the sixteen terms of $L_1$ (with coefficient $\mygamma^0$) form eight internal pairs. This pairing is done in a natural way with an eye towards rewriting our results in terms of $G(m,n,p)$ for specific values of $m$, $n$, and $p$. Collecting these twenty blocks modulo 5, we obtain
\begin{align}
\label{initial_pairings} 
\frac{L_0}{\mygamma}+L_1+\mygamma L_2 
& \equiv{}
\left(\frac{1}{\myalpha^{4}\mybeta^{4}\mygamma}-q^{6}\,\myalpha^{4}\mybeta^{4}\mygamma\right) \\
& \qquad\qquad +2q\left(\frac{\myalpha}{\mybeta^{4}\mygamma}+q^{4}\,\frac{\mybeta^{4}\mygamma}{\myalpha}\right) \notag\\
& \qquad\qquad  \vdots \notag\\
%+3q\left(\frac{1}{\myalpha\mybeta^{3}\mygamma}-q^{4}\,\myalpha\mybeta^{3}\mygamma\right)\notag\\
%&+2q\left(\frac{1}{\myalpha^{3}\mybeta^{2}\mygamma}+q^{4}\,\myalpha^{3}\mybeta^{2}\mygamma\right)
%+q^{2}\left(\frac{\myalpha^{4}}{\mybeta^{3}\mygamma}+q^{2}\,\frac{\mybeta^{3}\mygamma}{\myalpha^{4}}\right)
%+4q^{2}\left(\frac{\myalpha^{2}}{\mybeta^{2}\mygamma}-q^{2}\,\frac{\mybeta^{2}\mygamma}{\myalpha^{2}}\right)\notag\\
%&+2q^{2}\left(\frac{\mybeta}{\myalpha^{4}\mygamma}+q^{2}\,\frac{\myalpha^{4}\mygamma}{\mybeta}\right)
%+4q^{3}\left(\frac{\myalpha\mybeta}{\mygamma}-\frac{\mygamma}{\myalpha\mybeta}\right)
%+q^{3}\left(\frac{\mybeta^{2}}{\myalpha\mygamma}+\frac{\myalpha\mygamma}{\mybeta^{2}}\right)\notag\\
%&+4q^{3}\left(\frac{\mybeta^{3}}{\myalpha^{3}\mygamma}-\frac{\myalpha^{3}\mygamma}{\mybeta^{3}}\right)
%+3q^{2}\left(\frac{\mygamma}{\myalpha^{4}\mybeta^{2}}-q^{2}\,\frac{\myalpha^{4}\mybeta^{2}}{\mygamma}\right)
%+3q^{2}\left(\frac{\mygamma}{\myalpha^{2}\mybeta^{3}}+q^{2}\,\frac{\myalpha^{2}\mybeta^{3}}{\mygamma}\right)\notag\\
%&+3q\left(\frac{1}{\myalpha^{2}\mybeta^{4}}+q^{4}\,\myalpha^{2}\mybeta^{4}\right)
%+4q\left(\frac{1}{\myalpha^{4}\mybeta^{3}}-q^{4}\,\myalpha^{4}\mybeta^{3}\right)
%+q^{2}\left(\frac{\myalpha^{3}}{\mybeta^{4}}-q^{2}\,\frac{\mybeta^{4}}{\myalpha^{3}}\right)\notag\\
%&+3q^{2}\left(\frac{\myalpha}{\mybeta^{3}}+q^{2}\,\frac{\mybeta^{3}}{\myalpha}\right)
%+4q^{2}\left(\frac{1}{\myalpha\mybeta^{2}}-q^{2}\,\myalpha\mybeta^{2}\right)
%+2q^{2}\left(\frac{1}{\myalpha^{3}\mybeta}+q^{2}\,\myalpha^{3}\mybeta\right)\notag\\
& \qquad\qquad +3q^{3}\left(\frac{\myalpha^{4}}{\mybeta^{2}}+\frac{\mybeta^{2}}{\myalpha^{4}}\right) \notag\\
& \qquad\qquad +4q^{3}\left(\frac{\myalpha^{2}}{\mybeta}-\frac{\mybeta}{\myalpha^{2}}\right) \pmod5\notag\\
& \notag\\
\label{eq.paired-GF}& ={} q^{3}\,\bigl[G(1,2,3)+2\,G(-1,1,2)+3\,G(0,1,2)+2\,G(1,1,2)\\
&\quad+G(-2,0,1)+4\,G(-1,0,1)+2\,G(2,0,1)+4\,G(0,-1,0)\notag\\
&\quad+G(1,-1,0)+4\,G(2,-1,0)+3\,G(1,2,1)+3\,G(0,2,1)\notag\\
&\quad+3\,G(0,2,2)+4\,G(1,2,2)+G(-2,1,1)+3\,G(-1,1,1)\notag\\
&\quad+4\,G(0,1,1)+2\,G(1,1,1)+3\,G(-2,0,0)+4\,G(-1,0,0)\bigr],\notag
\end{align}
where the last equality follows from recognizing that each such pairing in \eqref{initial_pairings} can be rewritten thanks to \eqref{eq.G}. We next wish to simplify \eqref{eq.paired-GF} modulo 5. By finitely many applications of the recurrences \eqref{eq.recm}--\eqref{eq.recp} of
Lemma~\ref{thm.P1main}, together with the initial values \eqref{eq.init0}--\eqref{eq.init1} and the reflection property \eqref{eq.reflect}, we obtain the following equalities.
\begin{align}\label{eq.exactlist}
G(1,2,3)&=(U-1)(U-3)\left[\frac14\left(T-S+\frac9T\right)+\frac32
+(1-V)\left(K+2+\frac4K\right)\right]\notag\\
&\qquad{}-(U-2)\left[(1-V)\left(2+\frac9T\right)+\frac4K\right],\notag
\\
G(-1,1,2)&=(U-2)\left(K-2+\frac4K\right)+\frac{4(1-V)}{K}-2-\frac9T,\notag
\\
G(0,1,2)&=K(U-2)+V-1,\notag
\\
G(1,1,2)&=(U-2)\left(K+2+\frac4K\right)-2-\frac9T,\notag
\\
G(-2,0,1)&=(U-2)\left(1+\frac{16}{K^{2}}\right)-\frac1K\left(T-S+\frac9T\right)-\frac6K,\notag
\\
G(-1,0,1)&=\frac14\left(T-S+\frac9T\right)+\frac32-\frac{4(U-2)}{K},\notag
\\
G(2,0,1)&=(U-2)+\frac1K\left(T-S+\frac9T\right)+\frac6K,\notag
\\
G(0,-1,0)&=V-1,\notag
\\
G(1,-1,0)&=2+\frac9T-\frac{4(1-V)}{K},
\\
G(2,-1,0)&=-(1-V)\left(1+\frac{16}{K^{2}}\right)+\frac8K+\frac{36}{KT},\notag
\\[2pt]
G(1,2,1)&=\frac14\left(T-S+\frac9T\right)+\frac32+(1-V)\left(K+2+\frac4K\right),\notag
\\
G(0,2,1)&=(1-V)K+U-2,\notag
\\[2pt]
G(0,2,2)&=(1-V)K(U-2)-(1-V)^{2}+(U-2)^{2}-2,\notag
\\[2pt]
G(1,2,2)&=(U-2)\left[\frac14\left(T-S+\frac9T\right)+\frac32\right]\notag\\
&\qquad{}+(1-V)\left[(U-2)\left(K+2+\frac4K\right)-2-\frac9T\right]-\frac4K,\notag
\\
G(-2,1,1)&=K-4+\frac8K-\frac{16}{K^{2}},\notag\\
G(-1,1,1)&=K-2+\frac4K,\notag
\\
G(0,1,1)&=K,\notag\\
G(1,1,1)&=K+2+\frac4K,\notag
\\
G(-2,0,0)&=2+\frac{16}{K^{2}},\notag\\
G(-1,0,0)&=-\frac4K.\notag
\end{align}

We now reduce the twenty evaluations in \eqref{eq.exactlist} modulo $5$. Setting
\begin{equation}
A:=S+T^{-1}+4T,
\label{eq.AB}
\end{equation} we see that, modulo 5,
\begin{equation}\label{eq.exprA}
	\frac14\left(T-S+\frac9T\right)\equiv4T+S+T^{-1}=A, \quad \text{equivalently}\quad
	T-S+\frac9T\equiv4A.
\end{equation}
Using \eqref{eq.AB} and \eqref{eq.exprA}, the twenty equations in \eqref{eq.exactlist} reduce modulo $5$ to
\begin{align}\label{eq.exactlist_2}
G(1,2,3)&\equiv (U-1)(U-3)\left[A+4+(1-V)\left(K+2+\frac4K\right)\right]\notag\\
&\qquad{}+4(U-2)\left[(1-V)\left(2+\frac4T\right)+\frac4K\right],\notag\\
G(-1,1,2)&\equiv (U-2)\left(K+3+\frac4K\right)+\frac{4(1-V)}{K}+3+\frac1T,\notag\\
G(0,1,2)&\equiv K(U-2)+4(1-V),\notag\\
G(1,1,2)&\equiv (U-2)\left(K+2+\frac4K\right)+3+\frac1T,\notag\\
G(-2,0,1)&\equiv (U-2)\left(1+\frac1{K^{2}}\right)+\frac{A}{K}+\frac4K,\notag\\
G(-1,0,1)&\equiv A+4+\frac{U-2}{K},\notag\\
G(2,0,1)&\equiv (U-2)+\frac{4A}{K}+\frac1K,\notag\\
G(0,-1,0)&\equiv 4(1-V),\notag\\
G(1,-1,0)&\equiv 2+\frac4T+\frac{1-V}{K},\\
G(2,-1,0)&\equiv 4(1-V)\left(1+\frac1{K^{2}}\right)+\frac3K+\frac1{KT},\notag\\
G(1,2,1)&\equiv A+4+(1-V)\left(K+2+\frac4K\right),\notag\\
G(0,2,1)&\equiv (1-V)K+U-2,\notag\\
G(0,2,2)&\equiv (1-V)K(U-2)+4(1-V)^{2}+(U-2)^{2}+3,\notag\\
G(1,2,2)&\equiv (U-2)(A+4)+(1-V)\left[(U-2)\left(K+2+\frac4K\right)+3+\frac1T\right]+\frac1K,\notag
\\
G(-2,1,1)&\equiv K^{-2}\bigl(K^3+K^2+3K+4\bigr),\notag\\
G(-1,1,1)&\equiv K^{-1}\bigl(K^2+3K+4\bigr),\notag\\
G(0,1,1)&\equiv K,\notag\\
G(1,1,1)&\equiv K^{-1}(K^2+2K+4),\notag
\\
G(-2,0,0)&\equiv K^{-2}\bigl(2K^2+1\bigr),\notag\\
G(-1,0,0)&\equiv K^{-1} \pmod 5.\notag
\end{align}
Substituting the twenty congruences in \eqref{eq.exactlist_2}
into \eqref{eq.paired-GF} and collecting, we obtain
$$\frac{L_0}{\mygamma}+L_1+\mygamma L_2
\equiv q^3K^{-2}T^{-1}P(K,T,A,U,V)\pmod5,$$
and hence, from \eqref{eq.L-GF},
\begin{equation}
\sum_{n\geq0}\tbar(5n)q^n
\equiv\frac{f_5^{5}f_{10}^{5}f_{15}}{f_1^{6}f_2^{6}}q^3K^{-2}T^{-1}P(K,T,A,U,V)\pmod5,
\label{eq.P}
\end{equation}
where
\begin{align}\label{eq.P2}
P(K,T,A,U,V)={}&KT\bigl[KA+(K^{2}+2K+4)(1-V)+2K\bigr](U-2)^{2}\\
\notag&+\bigl[4K^{2}TA+K(2K^{2}T+KT+K+T)(1-V)\\
\notag&\qquad\qquad{}+2K^{3}T+2K^{2}T+KT+T\bigr](U-2)\\
\notag&+KT(K+4)A+2K^{2}T(1-V)^{2}+(4K^{2}+2KT+T)(1-V)\\
\notag&+2K^{2}T+3K^{2}+4KT+4K+2T .
\end{align}

Here and below we write
\begin{equation}\label{eq.g-def}
    g:=K-3-\frac4K =q^{-1}\dfrac{f_1^{2}f_2^{2}}{f_5^{2}f_{10}^{2}},
\end{equation}
where the last equality is from \eqref{eq.k}. 

\begin{lemma}\label{lem.param}
We have the following congruences modulo $5$:
\begin{align}
\label{eq.pK} K&\equiv\frac{(1-2qt)^{6}(1+2qt)^{2}}{qt\,(1+qt)^{3}(1+4qt)^{3}},\\
\label{eq.pT} T&\equiv\frac{(1-2qt)^{8}(1+qt)^{2}(1+4qt)^{2}}{q^{2}t^{2}\,(1+2qt)^{8}},\\
\label{eq.pg} g&\equiv\frac{1}{s^{8}\,qt\,(1-2qt)^{6}(1+qt)^{3}(1+2qt)^{2}(1+4qt)^{3}},\\
\label{eq.pU} U&\equiv\frac{1}{s^{4}\,qt\,(1-2qt)^{2}(1+qt)(1+2qt)^{2}(1+4qt)},\\
\label{eq.pV} V&\equiv\frac{1}{s^{4}\,(1-2qt)^{2}(1+qt)^{2}(1+2qt)^{2}(1+4qt)^{2}},\\
\label{eq.pS} S&\equiv\frac{1}{s^{12}\,q^{2}t^{2}\,(1-2qt)^{8}(1+qt)^{2}(1+2qt)^{8}(1+4qt)^{2}}.
\end{align}
\end{lemma}

\begin{proof}
By \eqref{bt}, we have
\begin{align}
\label{eq.e1} K&\equiv q^{-1}\frac{f_1^{24}}{f_2^{24}}, &
T&\equiv q^{-2}\frac{f_1^{24}}{f_3^{24}}, &
S&\equiv \frac{q^{-2}}{f_1^{12}f_3^{12}},\\
\label{eq.e2} U&\equiv \frac{q^{-1}}{f_1^{2}f_2^{2}f_3^{2}f_6^{2}}, &
V&\equiv \frac{f_1^{2}f_6^{6}}{f_2^{10}f_3^{6}}, &
g&\equiv \frac{q^{-1}}{f_1^{8}f_2^{8}}\pmod 5 .
\end{align}
Using $f_1, f_2, f_3,$ and $f_6$ from Lemma \ref{lem.f1236} completes the proof. 
\end{proof}

\begin{lemma}\label{lem.collapse}
We have
\begin{equation}
K^{-2}T^{-1}P(K,T,A,U,V)\equiv U\, g^{2} \pmod 5.
\label{eq.collapse}
\end{equation}
\end{lemma}

\begin{proof}
We first establish the following relation between $s$ and $t$:
\begin{equation}\label{eq.st}
\frac1{s^{4}}\equiv(1-2qt)^{6}(1+2qt)^{2}+qt\,(1+qt)^{3}(1+4qt)^{3}\pmod 5 .
\end{equation}
Indeed, by \eqref{eq.g-def} we have
\begin{equation}\label{eq.gK}
	g\equiv K+2+\frac1K=\frac{(1+K)^{2}}{K}\pmod 5.
\end{equation}
On the other hand, \eqref{eq.pK} gives
\[
1+K\equiv\frac{(1-2qt)^{6}(1+2qt)^{2}+qt\,(1+qt)^{3}(1+4qt)^{3}}{qt\,(1+qt)^{3}(1+4qt)^{3}}\pmod 5.
\]
Squaring the last congruence, dividing by \eqref{eq.pK}, and then using \eqref{eq.gK} together with \eqref{eq.pg} yields
\[
\left(\frac1{s^{4}}\right)^2\equiv\Bigl[(1-2qt)^{6}(1+2qt)^{2}+qt\,(1+qt)^{3}(1+4qt)^{3}\Bigr]^{2}\pmod 5.
\]
As $\mathbb{F}_5((q))$ is an integral domain, we get
\begin{equation}\label{eq.s4}
	\frac1{s^{4}}\equiv\pm\left((1-2qt)^{6}(1+2qt)^{2}+qt\,(1+qt)^{3}(1+4qt)^{3}\right)\pmod 5.
\end{equation}
Since $s=1+O(q)$ and $qt=q+O(q^{2})$, we have
\begin{align*}
	\frac1{s^{4}}=1+O(q)\qquad\text{and}\qquad
	(1-2qt)^{6}(1+2qt)^{2}+qt\,(1+qt)^{3}(1+4qt)^{3}=1+O(q).
\end{align*}
This implies that the sign in \eqref{eq.s4} is positive, which proves \eqref{eq.st}.  

We now prove \eqref{eq.collapse}.  The congruences \eqref{eq.pK}, \eqref{eq.pT}, \eqref{eq.pg}, and \eqref{eq.pU}  give
\[
K^{2}T\,U\,g^{2}\equiv\frac{(1-2qt)^{6}}{s^{20}\,q^{7}t^{7}\,(1+qt)^{11}(1+2qt)^{10}(1+4qt)^{11}}\pmod 5,
\]
so that \eqref{eq.collapse} is equivalent to
\[
P(K,T,A,U,V)\equiv\frac{(1-2qt)^{6}}{s^{20}\,q^{7}t^{7}\,(1+qt)^{11}(1+2qt)^{10}(1+4qt)^{11}}\pmod 5.
\]
Multiplying through by $q^{8}t^{8}(1+qt)^{11}(1+2qt)^{16}(1+4qt)^{11}$, it suffices to show
\begin{equation}\label{eq.polyw}
q^{8}t^{8}\,(1+qt)^{11}(1+2qt)^{16}(1+4qt)^{11}\,P(K,T,A,U,V)
\equiv \frac{qt\,(1-2qt)^{6}(1+2qt)^{6}}{s^{20}}\pmod 5.
\end{equation}
Replacing $A$ by $S+T^{-1}+4T$ in \eqref{eq.P2} and then substituting the congruences of Lemma \ref{lem.param} together with \eqref{eq.st} simplifies the left--hand side of \eqref{eq.polyw} into a polynomial in $qt$ of degree $53$ modulo $5$. Meanwhile, by \eqref{eq.st}, the right--hand side of \eqref{eq.polyw} is congruent to a polynomial in $qt$ of degree $53$ modulo $5$. One readily verifies that every coefficient of the difference between the two sides is divisible by $5$.
\end{proof}

Having established all the necessary tools to complete the proof of Theorem \ref{thm.main}, we now resume the proof from \eqref{eq.P}, where we left off. 

Applying Lemma \ref{lem.collapse} in \eqref{eq.P}, we obtain
\begin{align*}
\sum_{n\geq0}\tbar(5n)q^n
\equiv \frac{f_5^{5}f_{10}^{5}f_{15}}{f_1^{6}f_2^{6}}q^3U g^2 \pmod 5.
\end{align*}
Substituting \eqref{V-def} for $U$ and \eqref{eq.g-def} for $g$ into the above congruence, we get
\begin{align*}
    \sum_{n\geq0}\tbar(5n)q^n
&\equiv \frac{f_2f_3^{3}f_5}{f_1^{4}f_6^{2}} \pmod 5\\
&=\frac{f_1f_2f_3^3}{f_6^2}\cdot\frac{f_5}{f_1^5}\\
&\equiv \frac{f_1f_2f_3^3}{f_6^2} \pmod 5 \quad \text{(thanks to \eqref{bt})}.
\end{align*}
This completes the proof of Theorem \ref{thm.main}.
\end{proof}
Next, we establish the generating functions for $\tbar(40n)$ and $\tbar{(40n+20)}$ modulo $5$.
\begin{proposition}\label{prop.40n}
We have
\begin{align}
\sum_{n\ge0}\tbar(40n)q^{n}
&\equiv\frac{f_1^{6}}{f_3^{2}}\cdot\frac{f_8f_{12}^{2}}{f_2^{2}f_4f_{24}}
-4q\,\frac1{f_1f_3}\cdot\frac{f_2^{2}f_4^{2}f_6^{2}f_{24}}{f_8f_{12}} \pmod 5
\label{eq.40n},\\
\sum_{n\ge0}\tbar(40n+20)q^{n}
&\equiv4\frac1{f_1f_3}\cdot\frac{f_2^{3}f_6f_8f_{12}^{2}}{f_4f_{24}}
-\frac{f_1^{6}}{f_3^{2}}\cdot\frac{f_4^{2}f_6f_{24}}{f_2^{3}f_8f_{12}}
\pmod5.
\label{eq.40n20}
\end{align}
\end{proposition}
\begin{proof}
From \eqref{eq.main}, we have 
\begin{align*}
    \sum_{n\ge0}\tbar(5n)q^{n}
&\equiv\frac{f_2}{f_6^{2}}\cdot\frac{f_3^{3}}{f_1}\cdot f_1^{2}\pmod 5\\
&=\frac{f_2}{f_6^{2}}\left(\frac{f_4^{3}f_6^{2}}{f_2^{2}f_{12}}+q\,\frac{f_{12}^{3}}{f_4}\right)\left(\frac{f_2f_8^{5}}{f_4^{2}f_{16}^{2}}-2q\,\frac{f_2f_{16}^{2}}{f_8}\right) \quad \text{(using \eqref{eq.f1sq} and \eqref{eq.f33f1})}.
\end{align*}
Extracting the terms with even exponents of $q$ and then replacing $q^2$ by $q$ gives
\begin{align*}
    \sum_{n\ge0}\tbar(10n)q^{n}&\equiv \frac{f_2f_4^{5}}{f_6f_8^{2}}
-2q\frac{f_1^{2}f_6^{3}f_8^{2}}{f_2f_3^{2}f_4}\pmod5\\
&=\frac{f_2f_4^{5}}{f_6f_8^{2}}-2q\frac{f_6^{3}f_8^{2}}{f_2f_4}\left(\frac{f_2f_4^{2}f_{12}^{4}}{f_6^{5}f_8f_{24}}
-2q\frac{f_2^{2}f_8f_{12}f_{24}}{f_4f_6^{4}}\right) \quad \text{(thanks to \eqref{eq.f12f32})}.
\end{align*}
Isolating the terms in which the exponents of $q$ are even and then replacing $q^2$ by $q$, we obtain
\begin{align*}
     \sum_{n\ge0}\tbar(20n)q^{n}&\equiv \frac{f_1}{f_3}\left(\frac{f_2^{5}}{f_4^{2}}+4q\frac{f_4^{3}f_6f_{12}}{f_2^{2}}\right)
\pmod5\\
&=\left(\frac{f_2f_{16}f_{24}^{2}}{f_6^{2}f_8f_{48}}
-q\frac{f_2f_8^{2}f_{12}f_{48}}{f_4f_6^{2}f_{16}f_{24}}\right) \left(\frac{f_2^{5}}{f_4^{2}}+4q\frac{f_4^{3}f_6f_{12}}{f_2^{2}}\right) \quad \text{(using \eqref{eq.f1f3})}.
\end{align*}
Now, extracting the terms with even exponents of $q$ and then replacing $q^2$ by $q$ yields \eqref{eq.40n}. Similarly, isolating the terms with odd exponents of $q$, dividing both sides by $q$, and then replacing $q^2$ by $q$ gives \eqref{eq.40n20}.
\end{proof}

\begin{proof}[Proof of Theorem \ref{thm.20n}]
Substituting \eqref{eq.1f1f3} and \eqref{eq.f16f32} into \eqref{eq.40n}, we obtain
\begin{align*}
    \sum_{n\ge0}\tbar(40n)q^{n}
&\equiv\left(\frac{f_2f_4f_6^{5}}{f_{12}^{3}}
+4q\frac{f_2^2f_4^2f_{12}^2}{f_6^2}\right)\frac{f_8f_{12}^{2}}{f_2^{2}f_4f_{24}}\\
&\quad -4q\left(\frac{f_8^{2}f_{12}^{5}}{f_2^{2}f_4f_6^{4}f_{24}^{2}}
+q\frac{f_4^{5}f_{24}^{2}}{f_2^{4}f_6^{2}f_8^{2}f_{12}}\right)\frac{f_2^{2}f_4^{2}f_6^{2}f_{24}}{f_8f_{12}} \pmod 5\\
&=\frac{f_6^{5}f_8}{f_2f_{12}f_{24}}
-4q^{2}\,\frac{f_4^{7}f_{24}^{3}}{f_2^{2}f_8^{3}f_{12}^{2}}\\
&\equiv Y_0(q^2) \pmod 5.
\end{align*}
Similarly, substituting \eqref{eq.1f1f3} and \eqref{eq.f16f32} into \eqref{eq.40n20}, we get
\begin{align*}
    \sum_{n\ge0}\tbar(40n+20)q^{n}
&\equiv4\left(\frac{f_8^{2}f_{12}^{5}}{f_2^{2}f_4f_6^{4}f_{24}^{2}}
+q\frac{f_4^{5}f_{24}^{2}}{f_2^{4}f_6^{2}f_8^{2}f_{12}}\right)\frac{f_2^{3}f_6f_8f_{12}^{2}}{f_4f_{24}}\\
&\quad -\left(\frac{f_2f_4f_6^{5}}{f_{12}^{3}}
+4q\frac{f_2^2f_4^2f_{12}^2}{f_6^2}\right)\frac{f_4^{2}f_6f_{24}}{f_2^{3}f_8f_{12}}
\pmod5\\
&=4\frac{f_2f_8^{3}f_{12}^{7}}{f_4^{2}f_6^{3}f_{24}^{3}}
-\frac{f_4^{3}f_6^{6}f_{24}}{f_2^{2}f_8f_{12}^{4}}\\
&\equiv Y_1(q^2) \pmod 5.
\end{align*}
Finally, 
\begin{align*}
\sum_{n\ge0}\tbar(20n)q^{n}
=\sum_{n\ge0}\tbar(40n)q^{2n}+q\sum_{n\ge0}\tbar(40n+20)q^{2n}
\equiv Y_0(q^{4})+qY_1(q^{4})\pmod5. 
\end{align*}
\end{proof}

\section{Closing Thoughts}\label{sec.ct}
Note that Theorem \ref{thm.main} can also be used to provide a new elementary proof of the following theorem, first proven by Chern and Hao \cite{CH2019} via modular forms and later proved in elementary fashion by Lin, Liu, Wang, and Xiao \cite{LLWX}.  

\begin{theorem}
\label{thm:45n30}
For all $n\geq 0$, $\tbar(45n+30) \equiv 0 \pmod{5}$.  
\end{theorem}
In order to prove Theorem \ref{thm:45n30}, we require two additional 3-dissection results. 

\begin{lemma}
\label{lemma:3-dissections}
We have 
\begin{align*}
f_1f_2 &= \frac{f_6f_9^4}{f_3f_{18}^2} - qf_9f_{18} -2q^2\frac{f_3f_{18}^4}{f_6f_9^2}, \\
\frac{f_1^2}{f_2} &= \frac{f_9^2}{f_{18}} - 2q\frac{f_3f_{18}^2}{f_6f_9}.
\end{align*}
\end{lemma}
\begin{proof}
The first identity above appears in \cite{Chan}, while the second can be found in \cite[(14.3.2)]{Hirschhorn}.
\end{proof}
\begin{proof}[Proof of Theorem \ref{thm:45n30}]
Using Theorem \ref{thm.main} as our starting point, we have
\begin{align*}
\sum_{n\geq0}\tbar(5n)q^{n}
&\equiv
\frac{f_1f_2f_3^{3}}{f_6^{2}}\pmod5 \\
&= 
\frac{f_3^{3}}{f_6^{2}}(f_1f_2).
\end{align*} 
From Lemma \ref{lemma:3-dissections}, we then know 
\begin{align*}
\sum_{n\geq0}\tbar(15n)q^{3n}
&\equiv
\frac{f_3^{3}}{f_6^{2}}\left( \frac{f_6f_9^4}{f_3f_{18}^2} \right)  \pmod{5}
\end{align*} 
or 
\begin{align*}
\sum_{n\geq0}\tbar(15n)q^{n}
&\equiv
\frac{f_1^{2}f_3^4}{f_2f_6^{2}} \pmod{5} \\
&=  
\frac{f_3^4}{f_6^{2}}\left(  \frac{f_9^2}{f_{18}} - 2q\frac{f_3f_{18}^2}{f_6f_9} \right)
\end{align*} 
again from Lemma \ref{lemma:3-dissections}.  Note that, when expanded as a power series in $q$, the last expression above will contain no terms of the form $q^{3n+2}$.  Therefore, for all $n\geq 0$, 
$$
\tbar(15(3n+2)) = \tbar(45n+30) \equiv 0 \pmod{5}.
$$
\end{proof}

\bibliographystyle{abbrv}
%\nocite{*}
\bibliography{tbar.bib}

\end{document}